\documentclass{llncs}

\usepackage{amsmath,amsfonts,color,slashbox}
\newcommand{\eps}{\varepsilon}
\newcommand{\Ldiv}{L_{\text{div}}\,}
\newcommand{\Mdiv}{M_{\text{div}}}
\newcommand{\tnorm}[2]{|||#1|||_{#2}}
\newcommand{\LiStnorm}[1]{\tnorm{#1}{\text{div}}}
\def\Oh{{\mathcal{O}}}

\begin{document}

\title{First-order system least squares finite-elements for singularly
  perturbed reaction-diffusion equations\thanks{The work of S.M. was
    partially supported by NSERC discovery grants RGPIN-2014-06032 and
    RGPIN-2019-05692.}}

\titlerunning{FOSLS for singularly perturbed reaction-diffusion
  equations}

\author{James H. Adler\inst{1}\orcidID{0000-0002-6603-8840}  \and Scott MacLachlan\inst{2}\orcidID{0000-0002-6364-0684} \and Niall Madden\inst{3}\orcidID{0000-0002-4327-4234}}

\authorrunning{Adler, MacLachlan, and Madden}

\institute{Department of Mathematics, Tufts University
  \email{James.Adler@tufts.edu} \and Department of Mathematics and
  Statistics, Memorial University of Newfoundland
  \email{smaclachlan@mun.ca} \and School of Mathematics, Statistics,
  and Applied Mathematics, National University of Ireland Galway
  \email{Niall.Madden@NUIGalway.ie}}

\maketitle

\begin{abstract}
We propose a new first-order-system least squares (FOSLS) finite-element
discretization for singularly perturbed reaction-diffusion equations.
Solutions to such problems feature 
layer phenomena, and  are ubiquitous in many areas of applied
mathematics and modelling.
There is a long history
of the development of specialized numerical schemes for their accurate
numerical approximation. We follow a well-established practice of
employing \emph{a priori} layer-adapted meshes,
but with a novel finite-element method that yields a symmetric
formulation while also inducing a so-called ``balanced'' norm.
We prove continuity and coercivity of the FOSLS weak form, present
a suitable piecewise uniform mesh, and report on the results of
numerical experiments that demonstrate the accuracy and robustness of
the method.
\end{abstract}


\section{Introduction}
The numerical solution of singularly perturbed differential equations (SPDEs)
is of great interest to numerical analysts,  given the importance of
these equations in computational modelling, and the challenges they
present for classical numerical schemes and the mathematical methods
used to analyse them; see~\cite{RST} for a survey.
In this work, we focus on linear
second-order \emph{reaction-diffusion problems} of the form
\begin{equation}
  \label{eq:rd_eqn}
  -\eps \Delta u + b u = f \text{ on } \Omega:=(0,1)^d
  \qquad u|_{\partial \Omega}=0,
\end{equation}
for $d=1,2,3$, where we assume there exist constants
$0<b_0<b(\vec{x})<b_1$ for every $\vec{x}\in\Omega$.  
Like all SPDEs, \eqref{eq:rd_eqn} is
characterised by a small positive parameter that multiplies the
highest derivative. It is ``singular'' in the sense that the problem
is ill-posed if one formally sets $\eps=0$. As $\eps$ approaches this
limit, the  solution typically exhibits layers: regions of rapid
change, whose length is determined by $\eps$. The over-arching goal
is to devise methods that resolve these layers, and for which the
error (measured in a suitable norm) is independent of $\eps$.
Many classical techniques make the tacit assumption that derivatives
of $u$ are bounded, which does not hold,
uniformly in $\eps$, for solutions to \eqref{eq:rd_eqn}.
Numerous specialised methods, usually based around layer-adapted meshes,
have been developed with the goal of resolving these layers and the
attendant mathematical conundrums. The celebrated piecewise
uniform meshes of Shishkin have been particularly successful in this
regard; and analyses of finite-difference methods for
\eqref{eq:rd_eqn} and its many variants is largely
complete~\cite{MOS12}.

Finite-element methods (FEMs) applied on layer-adapted meshes have also been
successfully applied to \eqref{eq:rd_eqn}, but their analysis is more
problematic. This is highlighted to great effect by Lin and Stynes who
demonstrated that the usual energy norm associated with \eqref{eq:rd_eqn}
is too weak to adequately express the layers present in the
solution~\cite{LiSt12}. They proposed a first-order FEM (see
\S\ref{sec:balanced norms}) for which the
associated norm is sufficiently strong to capture layers; they coined
the term ``balanced norm'' to describe this. 

A flurry of activity on balanced norms was prompted by~\cite{LiSt12},
including the first-order system Petrov-Galerkin (FOSPeG) approach proposed by the
authors~\cite{JAdler_etal_2014d}, and we refer it its introduction for
a survey of the progress up to 2015. Since then, developments have
continued apace. Broadly speaking, studies can be classified as one of
two types.
\begin{enumerate}
  \item Those that give analyses of standard FEMs, but in norms that are not induced by
    the associated bilinear forms; see, e.g., \cite{RuSt19} on sparse
    grid FEMs, and \cite{MeXe16} on $hp$-FEMs.
  \item Those that propose new formulations  for which the associated norm is naturally 
   ``balanced''; see, e.g.,  the discontinuous Petrov-Galerkin method
   of Heuer and Karulik~\cite{HeKa17}.
 \end{enumerate}
 
The present study belongs to the second of these classes: we propose a new FEM
for which the induced norm is balanced. This method is related to our
earlier work~\cite{JAdler_etal_2014d}, but instead uses a weighted least-squares FEM to obtain a symmetric discrete system.  In this first-order system least-squares (FOSLS) approach \cite{CaLa94,CaMa97}, care is taken in choosing the weight, so that the resulting norms are indeed balanced.  

The remainder of the paper is outlined as follows.
Section~\ref{sec:balanced norms} gives a brief discussion on balanced
norms, where the Lin and Stynes and FOSPeG methods are summarized.
In Section~\ref{sec:FOSLS}, we discuss the weighted least-squares
approach and provide the necessary analysis, which applies in one, two
and three dimensions.
In Section~\ref{sec:numerics}, we focus on the particular case of
$d=2$; we present a suitable Shishkin mesh of the problem, and present
numerical results that support our findings.
Some concluding remarks are given in Section~\ref{sec:conclusions}.     


\section{Balanced norms}
\label{sec:balanced norms}
In \cite{LiSt12}, Lin and Stynes propose a first-order system
reformulation of \eqref{eq:rd_eqn}, writing the equivalent
system as
  \begin{equation}\label{eq:L1}%
    \Ldiv \mathcal{U} :=
    \begin{pmatrix}
      \eps^{1/4}\big(\vec{w} - \nabla u \big) \\
      -\eps \nabla \cdot \vec{w} + bu
    \end{pmatrix}
    =
    \begin{pmatrix}
      0 \\
      f
    \end{pmatrix}
    =: \mathcal{F}_{\text{div}},
\end{equation}
for $\mathcal{U}=(u, \vec w)^T$. Rather than forming a least-squares
finite-element discretization as in \cite{CaLa94,CaMa97}, they choose
to close the system in a nonsymmetric manner, defining $\mathcal{V} = (v,\vec z)^T$ and
\[
  \Mdiv\mathcal{V} :=
  \begin{pmatrix}
    \eps^{1/4}\big(\vec{z} - \nabla v \big)\\
    -\eps^{1/2}b^{-1} \nabla \cdot \vec{z} + v
  \end{pmatrix},
\]
then writing the solution of \eqref{eq:rd_eqn} as that of the weak form
\begin{equation}
\label{eq:LSweak}
a_{\text{div}}(\mathcal{U},\mathcal{V}) :=
\langle \Ldiv
\mathcal{U},\Mdiv\mathcal{V}\rangle
= \langle \mathcal{F}_{\text{div}},\Mdiv\mathcal{V}\rangle
\quad \forall \mathcal{V} \in H^1(\Omega)\times H(\text{div}).
\end{equation}
In \cite{LiSt12}, it is shown that $a_{\text{div}}$ is coercive and
continuous with respect to the norm,
\begin{equation}
\label{eq:LiSt-norm}
\LiStnorm{\mathcal{U}}^2:=
b_0 \|u\|_0^2 +\frac{\eps^{1/2}}{2} \|\nabla u\|_0^2  + 
\frac{\eps^{1/2}}{2} \|\vec w\|_0^2 + 
\eps^{3/2}\|\nabla \cdot \vec w\|_0^2,
\end{equation}
which is shown to be a \emph{balanced} norm for the problem, in the
sense that all the components in \eqref{eq:LiSt-norm} have the same order of
magnitude.

In \cite{JAdler_etal_2014d}, the authors augmented the first-order
system approach proposed by Lin and Stynes to include a curl
constraint, in the same style as \cite{CaMa97}, leading to the
first-order system reformulation of \eqref{eq:rd_eqn} as
\begin{equation}\label{eq:L}%
  L\,\mathcal{U}:=
  \begin{pmatrix}
  {\eps}^{1/4}\big(\vec w - \nabla u \big) \\
  -\eps \nabla \cdot \vec w + bu \\
  \eps \nabla \times \vec w
\end{pmatrix}
=
\begin{pmatrix}
\vec{0}\\
f \\
\vec{0}
\end{pmatrix}
=: \hat{\mathcal{F}}.
\end{equation}
Then, writing
\begin{equation}\label{eq:Mk}%
  M_k\mathcal{V} :=
  \begin{pmatrix}
    {\eps}^{1/4}\big(\vec z - \nabla v \big)\\
  -\eps^{1/2} b^{-1}\nabla \cdot \vec z + v\\
  \eps^{k/2} \nabla \times \vec z
\end{pmatrix},
\end{equation}
leads to the weak form
\begin{equation}
\label{eq:balanced weak form}
 a_k^{}(\mathcal{U},\mathcal{V}) := \langle L\mathcal{U},M_k\mathcal{V}\rangle =
\langle\hat{\mathcal{F}},M_k\mathcal{V}\rangle \quad \forall\mathcal{V} \in
\left(H^1(\Omega)\right)^{1+d}.
\end{equation}
Building on the theory of \cite{LiSt12}, this form is shown to be
coercive and continuous with respect to the balanced norm
\begin{equation}
  \tnorm{\mathcal{U}}{k}^2 = b_0 \|u\|_0^2 +\frac{\eps^{1/2}}{2} \|\nabla u\|_0^2  + 
\frac{\eps^{1/2}}{2} \|\vec w\|_0^2 + 
\eps^{3/2}\|\nabla \cdot \vec w\|_0^2 + \eps^{1+k/2}\|\nabla \times
\vec w\|_0^2.
\end{equation}
Furthermore, in \cite{JAdler_etal_2014d}, the authors show that, when
discretized using piecewise bilinear finite elements on a
tensor-product Shishkin mesh, this weak form leads to a {\it
  parameter-robust} discretization, with an error estimate independent
of the perturbation parameter $\eps$.

\section{First-order system least squares finite-element methods}
\label{sec:FOSLS}
While theoretical and numerical results in \cite{JAdler_etal_2014d}
show the effectiveness of the first-order system Petrov-Galerkin approach proposed therein, the
non-symmetric nature of the weak form also has disadvantages.  Primary
among these is that the weak form no longer can be used as an accurate
and reliable error indicator, contrary to the common practice for
FOSLS finite-element approaches \cite{1997BerndtM_ManteuffelT_McCormickS-aa,2012BrezinaM_GarciaJ_ManteuffelT_McCormickS_RugeJ_TangL-aa,CaLa94,CaMa97,2008De-SterckH_ManteuffelT_McCormickS_NoltingJ_RugeJ_TangL-aa}.
Standard techniques to symmetrize the weak form in
\eqref{eq:balanced weak form} fail, however, either sacrificing the
balanced nature of the norm (and, thus, any guarantee of parameter
robustness of the resulting discretization) or coercivity or
continuity of the weak form (destroying standard error estimates).
Here, we propose a FOSLS approach for the problem in
\eqref{eq:rd_eqn}, made possible by considering a weighted norm with
spatially varying weight function.  Weighted least-squares formulations have been used for a wide variety of problems including those with singularities due to the domain \cite{Manteuffel2006,Manteuffel2008}.

To this end, we define the weighted inner product on both scalar and
vector $H^1(\Omega)$ spaces, writing
\[
\langle u,v \rangle_\beta = \int_\Omega \beta(\vec{x}) u(\vec{x})
v(\vec{x})\, d\vec{x},
\]
with the associated norm written as $\|u\|_\beta$.  Slightly
reweighting the first-order system from \eqref{eq:L}, we have
\begin{equation}\label{eq:FOSLS L}%
  \mathcal{L}\,\mathcal{U} :=
  \begin{pmatrix}
    {\eps}^{1/2}\big(\vec{w} - \nabla u \big)\\
    -\eps b^{-1/2}\nabla \cdot \vec{w} + b^{1/2}u \\
    \eps^{k/2} \nabla \times \vec{w}
  \end{pmatrix}
=
\begin{pmatrix}
\vec{0}\\
b^{-1/2}f \\
\vec{0}
\end{pmatrix}
=: \mathcal{F}.
\end{equation}
and pose the weighted FOSLS weak form as
\[
a(\mathcal{U},\mathcal{V}) = \langle
\mathcal{L}\mathcal{U},\mathcal{L}\mathcal{V} \rangle_\beta = \langle
\mathcal{F},\mathcal{L}\mathcal{V}\rangle_\beta  \quad \forall\mathcal{V} \in
\left(H^1(\Omega)\right)^{1+d}.
\]
This form leads to a natural weighted product norm given by
\[
\tnorm{\mathcal{U}}{\beta,k}^2 =
\|u\|_\beta^2 + \eps\|\nabla u\|_\beta^2 + \eps\|\vec{w}\|_\beta^2 +
\eps^2\|\nabla\cdot\vec{w}\|_\beta^2 + \eps^k\|\nabla\times\vec{w}\|_\beta^2.
\]
As shown below, under a reasonable assumption on the weight function,
$\beta$, the FOSLS weak form is coercive and continuous with respect
to this norm.

\begin{theorem}
  Let $\beta(\vec{x})$ be given such that there exists $C>0$ for which
  \[
\nabla\beta \cdot \nabla\beta < \frac{b_0\beta^2(\vec{x})}{\eps(1+C)^2},
  \]
  for every $\vec{x}\in\Omega$, and let $k\in\mathbb{R}$ be given.  Then,
  \begin{align*}
  |a(\mathcal{U},\mathcal{V})| \leq & \left(3+2\max(b_0^{-1},b_1)\right)\tnorm{\mathcal{U}}{\beta,k}\tnorm{\mathcal{V}}{\beta,k}\\
 \min\left(\frac{C\min(1,b_0)}{1+C},b_1^{-1},1\right)\tnorm{\mathcal{U}}{\beta,k}^2 \leq & a(\mathcal{U},\mathcal{U})
  \end{align*}
  for all $\mathcal{U},\mathcal{V}\in \left(H^1(\Omega)\right)^{1+d}$.
\end{theorem}
\begin{proof}
  For the continuity bound, we note that
  \begin{align*}
a(\mathcal{U},\mathcal{V}) = \eps &\langle \vec{w}-\nabla
u,\vec{z}-\nabla v \rangle_\beta\\ & + \langle
-\eps b^{-1/2}\nabla\cdot\vec{w}+b^{1/2}u,-\eps b^{-1/2}\nabla\cdot \vec{z}
+ b^{1/2}v\rangle_\beta \\ & +
\eps^k\langle\nabla\times\vec{w},\nabla\times\vec{z}\rangle_\beta.
\end{align*}
Thus, by the Cauchy-Schwarz and triangle inequalities, we have
\begin{align*}
  |a(\mathcal{U},\mathcal{V})| \leq &
  \eps\left( \|\vec{w}\|_\beta + \|\nabla u\|_\beta\right)\left(
  \|\vec{z}\|_\beta + \|\nabla v\|_\beta\right)\\& +
  \left(\eps b_0^{-1/2}\|\nabla\cdot \vec{w}\|_\beta + b_1^{1/2}\|u\|_\beta\right)
  \left(\eps b_0^{-1/2}\|\nabla\cdot \vec{z}\|_\beta + b_1^{1/2}\|v\|_\beta\right)
  \\&+ \eps^k\|\nabla\times\vec{w}\|_\beta\|\nabla\times\vec{z}\|_\beta
  \\
  \leq & \left(3+2\max(b_0^{-1},b_1)\right)\tnorm{\mathcal{U}}{\beta,k}\tnorm{\mathcal{V}}{\beta,k}.
\end{align*}

For the coercivity bound, we note
\begin{align*}
a(\mathcal{U},\mathcal{U}) =  \eps\|\vec{w}-\nabla u\|_\beta^2 &+
\eps^2\|b^{-1/2}\nabla\cdot\vec{w}\|_\beta^2 +\|b^{1/2}u\|_\beta^2 +
\eps^k\|\nabla\times\vec{w}\|_\beta^2 \\&-
2\eps\langle\nabla\cdot\vec{w},u\rangle_\beta\\
\geq \eps\|\vec{w}-\nabla u\|_\beta^2 &+
\eps^2b_1^{-1}\|\nabla\cdot\vec{w}\|_\beta^2 +b_0\|u\|_\beta^2 +
\eps^k\|\nabla\times\vec{w}\|_\beta^2 \\&- 2\eps\langle\nabla\cdot\vec{w},u\rangle_\beta.
\end{align*}
Now consider
\begin{align*}
- 2\eps\langle\nabla\cdot\vec{w},u\rangle_\beta & = -2\eps\int_\Omega
\left(\nabla\cdot\vec{w}\right)u\beta d\vec{x} \\
& = 2\eps\int_\Omega\vec{w}\cdot\nabla(u\beta)d\vec{x} \\
& = 2\eps\int_\Omega \left(\vec{w}\cdot\nabla u\right)\beta d\vec{x} +
2\eps \int_\Omega\left(\nabla\beta \cdot \vec{w}\right)u d\vec{x}\\
& = 2\eps\langle \vec{w},\nabla u\rangle_\beta + 2\eps \int_\Omega\left(\nabla\beta \cdot \vec{w}\right)u d\vec{x},
\end{align*}
where we use the fact that $u=0$ on the boundary in the integration by parts step.
Note that
\[
\langle \vec{w},\nabla u\rangle_\beta = \frac{1}{4}\|\vec{w}+\nabla
u\|_\beta^2 - \frac{1}{4}\|\vec{w}-\nabla u\|_\beta^2,
\]
and, consequently, that
\[
\eps\|\vec{w}-\nabla u\|_\beta^2 + 2\eps\langle\vec{w},\nabla
u\rangle_\beta = \frac{\eps}{2}\|\vec{w}+\nabla
u\|_\beta^2 + \frac{\eps}{2}\|\vec{w}-\nabla u\|_\beta^2
= \eps\|\vec{w}\|_\beta^2 + \eps\|\nabla u\|_\beta^2.
\]
Thus,
\begin{align*}
a(\mathcal{U},\mathcal{U}) \geq b_0\|u\|_\beta^2 &+ \eps\|\vec{w}\|_\beta^2 + \eps\|\nabla u\|_\beta^2+
\eps^2b_1^{-1}\|\nabla\cdot\vec{w}\|_\beta^2 +
\eps^k\|\nabla\times\vec{w}\|_\beta^2 \\ &+ 2\eps\int_\Omega\left(\nabla\beta \cdot \vec{w}\right)u d\vec{x}.
\end{align*}
Finally, consider
\[
2\eps\left|\int_\Omega\left(\nabla\beta \cdot \vec{w}\right)u d\vec{x}\right| =
2\eps\left|\left\langle\vec{w},\frac{u}{\beta}\nabla\beta\right\rangle_\beta\right| \leq 2\eps\|\vec{w}\|_\beta\left\|\frac{u}{\beta}\nabla\beta\right\|_\beta.
\]
By our assumption on $\beta$,
\[
\left\|\frac{u}{\beta}\nabla\beta\right\|_\beta^2 \leq \frac{b_0}{\eps(1+C)^2}\|u\|_\beta^2,
\]
and, so,
\[
2\eps\left|\int_\Omega\left(\nabla\beta \cdot \vec{w}\right)u
d\vec{x}\right| \leq 2\frac{\eps^{1/2}b_0^{1/2}}{1+C}\|\vec{w}\|_\beta\|u\|_\beta.
\]
This gives
\begin{align*}
a(\mathcal{U},\mathcal{U}) \geq &b_0\|u\|_\beta^2 + \eps\|\vec{w}\|_\beta^2 + \eps\|\nabla u\|_\beta^2+
\eps^2b_1^{-1}\|\nabla\cdot\vec{w}\|_\beta^2 +
\eps^k\|\nabla\times\vec{w}\|_\beta^2\\ & -
2\frac{\eps^{1/2}b_0^{1/2}}{1+C}\|\vec{w}\|_\beta\|u\|_\beta\\
\geq & b_0\left(1-\frac{1}{(1+C)}\right)\|u\|_\beta^2 + \eps\left(1-\frac{1}{(1+C)}\right)\|\vec{w}\|_\beta^2 \\&+ \eps\|\nabla u\|_\beta^2+
\eps^2b_1^{-1}\|\nabla\cdot\vec{w}\|_\beta^2 +
\eps^k\|\nabla\times\vec{w}\|_\beta^2\\
\geq & \min\left(\frac{C\min(1,b_0)}{1+C},b_1^{-1},1\right)\tnorm{\mathcal{U}}{\beta,k}^2.
\end{align*}
\end{proof}

A natural question, in light of this result, is whether a suitable
choice of $\beta(\vec{x})$ exists.  We now
give a concrete construction of one such family of functions,
$\beta(\vec{x})$, for which the assumption above is satisfied.
This family is constructed for the case of $\Omega = [0,1]^d$ with
boundary layers along each boundary adjacent to the origin (i.e., where $x_i = 0$ for some $i$).
The extension to boundary layers along all $2d$ boundary faces is
straightforward from the construction.

\begin{theorem}
Let $C>0$ be given, and define $\gamma = \frac{b_0^{1/2}}{(1+C)\sqrt{d}}$.
Take
\begin{equation}
  \label{eq:def beta}
\beta(\vec{x}) = \left(1+\frac{1}{\sqrt{\eps}}e^{-\gamma x_1/\sqrt{\eps}}\right)\cdots\left(1+\frac{1}{\sqrt{\eps}}e^{-\gamma x_d/\sqrt{\eps}}\right)
\end{equation}
Then,
\[
\nabla\beta \cdot \nabla\beta < \frac{b_0\beta^2(\vec{x})}{\eps(1+C)^2},
  \]
  for every $\vec{x}\in\Omega$.

\end{theorem}
\begin{proof}
  A direct calculation shows that
  \[
\frac{\partial\beta}{\partial x_i} = \frac{\frac{-\gamma}{\eps}e^{-\gamma x_i/\sqrt{\eps}}}{\left(1+\frac{1}{\sqrt{\eps}}e^{-\gamma x_i/\sqrt{\eps}}\right)}\beta(\vec{x}).
\]
Consequently,
\[
\nabla\beta\cdot\nabla\beta = \sum_{i=1}^d \left(\frac{\frac{-\gamma}{\eps}e^{-\gamma x_i/\sqrt{\eps}}}{\left(1+\frac{1}{\sqrt{\eps}}e^{-\gamma x_i/\sqrt{\eps}}\right)}\right)^2\beta^2(\vec{x}).
\]
Note, however, that
\[
\left(\frac{\frac{-\gamma}{\eps}e^{-\gamma
    x_i/\sqrt{\eps}}}{\left(1+\frac{1}{\sqrt{\eps}}e^{-\gamma
    x_i/\sqrt{\eps}}\right)}\right)^2 = \frac{\gamma^2}{\eps}\left(\frac{\frac{1}{\sqrt{\eps}}e^{-\gamma
    x_i/\sqrt{\eps}}}{\left(1+\frac{1}{\sqrt{\eps}}e^{-\gamma
    x_i/\sqrt{\eps}}\right)}\right)^2 \leq \frac{\gamma^2}{\eps}.
\]
This gives
\[
\nabla\beta\cdot\nabla\beta \leq \frac{d\gamma^2}{\eps} \beta^2(\vec{x}).
\]
Substituting in the chosen value for $\gamma$ gives the stated result.
\end{proof}

The final question to be resolved is whether
$\beta(\vec{x})$ as given in \eqref{eq:def beta} is a ``good''
choice, in the sense of whether quasi-optimal
approximation in the resulting norm is expected to give a good
approximation to the layer structure in a typical solution. We
consider the case of $d=2$, the unit square.  Following Lemmas 1.1
and 1.2 of \cite{LiMa09a}, we require that
the problem data satisfy the  assumptions of \cite[\S
2.1]{LiMa09a}, specifically that $f,b \in C^{4,\alpha}(\bar{\Omega})$
and that $f$ vanishes at the corners of the domain.  Denoting the four
edges of the domain by $\Gamma_i$, $1\leq i \leq 4$, numbered
clockwise with the edge $y=0$ as $\Gamma_1$, and the four corners of
the domain by $c_i$, $1\leq i \leq 4$, numbered clockwise with the
origin as $c_1$, we have the following result.

\begin{lemma}[{\cite[Lemmas 1.1 and 1.2]{LiMa09a}}]
 \label{lem:decomp}
 The solution $u$ of \eqref{eq:rd_eqn} can be decomposed as
\begin{subequations}
  \label{eq:decomp}
 \begin{equation}
  u = V+W+Z = V + \displaystyle\sum\limits_{i=1}^4 W_i
  + \displaystyle\sum\limits_{i=1}^4 Z_i,
 \end{equation}
 where each $W_i$ is a layer associated with the edge $\Gamma_i$ and each $Z_i$ is a layer
 associated with the corner $c_i$. 
There exists a constant $C$ such that
 \begin{align}
   \label{eq:V}
   \left|\frac{\partial^{m+n}V}{\partial x^m \partial y^n}
     (x,y)\right| &    \leq C(1+\eps^{1-m/2-n/2}),&     
  0\leq m+n \leq 4,&\\
  \label{eq:W1}
   \left|\frac{\partial^{m+n}W_1 }{\partial x^m \partial y^n}
     (x,y)\right|   &\leq C(1+\eps^{1-m/2})\eps^{-n/2} e^{-y\sqrt{b_0/(2\eps)}}, &
   0\leq m+n \leq 3,&\\
  \label{eq:W2}
   \left|\frac{\partial^{m+n}W_2 }{\partial x^m \partial y^n} (x,y)\right| 
&\leq C\eps^{-m/2}(1+\eps^{1-n/2}) e^{-x\sqrt{b_0/(2\eps)}}, &
 0\leq m+n \leq 3,&\\
\label{eq:Z1}
 \left|\frac{\partial^{m+n} Z_1}{\partial x^m \partial y^n }
 (x,y)\right| &\leq C \eps^{-m/2-n/2}e^{-(x+y)\sqrt{b_0/(2\eps)}},
 & 0\leq m+n \leq 3,&
\end{align}
\end{subequations}
with analogous bounds for $W_3$, $W_4$,  $Z_2$, $Z_3$ and $Z_4$.
\end{lemma}

Thus, as a ``stereotypical'' solution of \eqref{eq:rd_eqn} in the case
where boundary layers only form along the edges $x=0$ and $y=0$ of
$[0,1]^2$, we can consider
\[
u(x) = u_0(x) + c_1e^{-x\sqrt{b_0/(2\eps)}} +
c_2e^{-y\sqrt{b_0/(2\eps)}} + c_3e^{-(x+y)\sqrt{b_0/(2\eps)}}.
\]
Next, we check if $\tnorm{\mathcal{U}}{\beta,k}$ is ``balanced'',
not only in the sense of all terms having the same order, but in
addition that
each component in the stereotypical solution above is well-represented
in the norm.  This means the norm can be bounded from above and below by
$\eps$-independent values, so that it is not seen as being
well-approximated by zero in the norm (unless truly vanishingly
small), nor that the norm blows up as $\eps \rightarrow 0$.
For this case, \eqref{eq:def beta} simplifies as
\begin{equation*}
  \beta(x,y) = \beta_1(x)\beta_1(y)
  \quad \text{ where } \quad \beta_1(x) = 1+\frac{1}{\sqrt{\eps}}e^{-\gamma
    x_1/\sqrt{\eps}},
\end{equation*}
and the checks rely on two direct calculations:
\begin{align*}
\int_0^1 \beta_1(x)dx & = 1 +
\frac{1}{\gamma}\left(1-e^{-\gamma/\sqrt{\eps}}\right) \approx 1 +
\frac{1}{\gamma}, \\
\int_0^1 \beta_1(x) \left(e^{-x\sqrt{b_0/(2\eps)}}\right)^2dx & =
\frac{1}{\gamma+\sqrt{2b_0}}\left(1-e^{-\gamma/\sqrt{\eps} -
                                                                \sqrt{2b_0/\eps}}\right) \\
  & \qquad \qquad +
\sqrt{\frac{\eps}{2b_0}}\left(1-e^{-2\sqrt{b_0/(2\eps)}}\right)\\
& \approx \frac{1}{\gamma+\sqrt{2b_0}},
\end{align*}
With this, assuming that $u_0(x)$ is $\Oh(1)$ over
a nontrivial fraction of the domain, we conclude that
\[
\tnorm{(u_0,\nabla u_0)^T}{\beta,k} \approx 1+\frac{1}{\gamma},
\]
because of the separable nature of the calculation.  Thus, the regular
part of the solution is well-represented in the norm.

For the $W_2$ layer term, we write $w_2(x,y) = e^{-x\sqrt{b_0}/2\eps}$
and calculate from the above that
\[
\left\|w_2\right\|_\beta^2 \approx \left(1+\frac{1}{\gamma}\right)\frac{1}{\gamma+\sqrt{2b_0}}.
\]
Noting that all derivatives of this term with respect to $y$ are zero
and that $\partial_x^\ell w_2 = (-\sqrt{b_0/(2\eps)})^\ell w_2$, we
compute
\begin{align*}
\tnorm{(w_2,\nabla w_2)^T}{\beta,k}^2 & = 
\|w_2\|_\beta^2 + \eps\|\nabla w_2\|_\beta^2 + \eps\|\nabla w_2\|_\beta^2 +
                                        \eps^2\|\nabla\cdot \nabla w_2\|_\beta^2 \\
  & \qquad \quad + \eps^k\|\nabla\times\nabla
w_2\|_\beta^2 \\
& = \|w_2\|_\beta^2 + \frac{b_0}{2}\|w_2\|_\beta^2 + \frac{b_0}{2}\|w_2\|_\beta^2 +
\left(\frac{b_0}{2}\right)^2\|w_2\|_\beta^2+0 \\
& \approx \left(1+b_0 + \left(\frac{b_0}{2}\right)^2\right)\left(1+\frac{1}{\gamma}\right)\frac{1}{\gamma+\sqrt{2b_0}}.
\end{align*}
Again, this shows that the $W_2$ layer term is well-represented in the
norm.  Similar calculations show the same to be true for the $W_1$
layer and $Z_1$ corner terms in the stereotypical solution.

\section{Numerical Results}
\label{sec:numerics}

To test the above approach, we consider a two-dimensional problem
with constant $b=1$ posed on the unit square. We construct a problem
whose solution mimics the stereotypical solution discussed above,
with two edge layers and one corner layer.  Specifically, we choose
$f$ so that the
solution is
\[
u(x,y) = \left(\cos\left(\frac{\pi x}{2}\right)- \frac{e^{-x/\sqrt{\eps}} -e^{-1/\sqrt{\eps}}}{1-e^{-1/\sqrt{\eps}}}\right)
  \left(1-y- \frac{e^{-y/\sqrt{\eps}}-e^{-1/\sqrt{\eps}}}{1-e^{-1/\sqrt{\eps}}}\right).
\]
We note that this has somewhat more complex layer behaviour than the
stereotypical solution, but still obeys the bounds of Lemma
\ref{lem:decomp}.  Also, the solution is constructed so as to
obey the homogeneous Dirichlet boundary conditions.  For numerical
stability, we rescale the equations by defining $\vec{w} =
\sqrt{\eps}\nabla u$ and making corresponding changes in weights to
preserve the balanced nature of the norm.  With this, we pick $k$ to
match the powers of $\eps$ in the weighting terms of both
$\|\nabla\cdot\vec{w}\|_\beta^2$ and
$\|\nabla\times\vec{w}\|_\beta^2$ in $\tnorm{\mathcal{U}}{\beta,k}$,
equivalent to taking $k=2$ above.

We discretize the test problem on a tensor-product Shishkin mesh (see,
e.g.,~\cite[\S3]{JAdler_etal_2014d} for more details).  To
do this, we select a transition point, $\tau > 0$, and construct a
one-dimensional mesh with $N/2$ equal-sized elements on each of the intervals
$[0,\tau]$ and  $[\tau,1]$.
The two-dimensional mesh is created as a tensor-product of this mesh
with itself, with rectangular (quadrilateral) elements.  For the
choice of $\tau$, we slightly modify the standard choice from the
literature (see, for example, \cite{JAdler_etal_2014d,LiSt12,LiMa09a})
to account for both the layer functions present in the solution
decomposition and in the definition of $\beta(\vec{x})$ in
\eqref{eq:def beta}.  As such, we
take
\[
\tau = \min\left\{\frac{1}{2}, (p+1)\sqrt{\frac{2\eps}{b_0}}\gamma^{-1}\ln N\right\}
\]
where $p$ is the degree of the polynomial space ($p=1$ for bilinear
elements, $p=2$ for biquadratic, and $p=3$ for bicubic), so that this factor
matches the expected $L^2$ rate of convergence of the approximation, while the terms $\sqrt{{2\eps}/{b_0}}\gamma^{-1}$
decrease appropriately as $\eps$ does, but increase (corresponding to
increasing layer width) with decreases in $b_0$ or $\gamma$.  In the
results that follow, we take $\gamma = 0.5$, implying $C =
\sqrt{2}-1$.  All numerical results were computed using
Firedrake \cite{rathgeber2017firedrake} for the discretization and a
direct solver for the resulting linear systems.

Table \ref{tab:expected} shows the expected reduction rates in errors
with respect to the mesh parameter, $N$, if we were to have standard
estimates of approximation error in the $\beta$-norm on the Shishkin
meshes considered here.  Tables \ref{tab:bilinear}, \ref{tab:biquad} and
\ref{tab:bicubic} show the measured errors (relative to the
manufactured solution) for the bilinear, biquadratic, and bicubic
discretizations, respectively.  Expected behaviour for the bilinear case
is a reduction like $N^{-1}\ln N$ for
$\tnorm{\mathcal{U}^\ast-\mathcal{U}^N}{\beta,2}$ (where
$\mathcal{U}^\ast$ represents the manufactured solution, $u^\ast$ and
its gradient) and like $(N^{-1}\ln N)^2$ for the discrete maximum norm
of the error, $\|u^*-u^N\|_{\ell_{\infty}}$, which is measured at the
nodes of the mesh corresponding to the finite-element degrees of
freedom.  These are both expected to be raised by one power in the
biquadratic case, and a further one power for bicubics.  In Tables \ref{tab:bilinear},
\ref{tab:biquad}, and \ref{tab:bicubic}, we see convergence behaviour comparable to
these rates, with the exception of the results for the discrete
maximum norm in Table \ref{tab:biquad}. These seem to show a
superconvergence-type phenomenon, although we have no
explanation for this observation at present.

\begin{table}[p]
  \centering
  \caption{Expected error reduction rates  on a Shishkin mesh.}
  \label{tab:expected}
  \begin{tabular}{c|cccc}
    & ~$N=64$~ & ~$N=128$~ & ~$N=256$~ & ~$N=512$~\\
    \hline
    $N^{-1}\ln N$  & 0.60 & 0.58 & 0.57 & 0.56\\
    $(N^{-1}\ln N)^2$ & 0.36 & 0.34 & 0.33 & 0.32\\
    $(N^{-1}\ln N)^3$ & 0.22 & 0.20 & 0.19 & 0.18\\
    $(N^{-1}\ln N)^4$ & 0.13 & 0.12 & 0.11 & 0.10
  \end{tabular}
\end{table}

\begin{table}[p]
\centering
\caption{$\beta$-weighted norm and discrete max norm errors for model
  problem with bilinear discretization.}
\label{tab:bilinear}
\begin{tabular}{|c|ccccc|}
\hline
&\multicolumn{5}{|c|}{$\tnorm{\mathcal{U}^\ast-\mathcal{U}^N}{\beta,2}$  (Reduction Rate w.r.t. N)} \\
\hline
$\eps$/$N$&32&64&128&256&512\\
\hline
$10^{-6}$  &  3.086e-01 & 1.921e-01 (0.62) & 1.137e-01 (0.59)&
6.531e-02 (0.57) & 3.680e-02 (0.56)\\
$10^{-8}$  &  3.086e-01 & 1.921e-01 (0.62)& 1.137e-01 (0.59)&
6.532e-02 (0.57) & 3.681e-02 (0.56)\\
$10^{-10}$  &  3.086e-01 & 1.921e-01 (0.62)& 1.137e-01 (0.59)&
6.533e-02 (0.57) & 3.681e-02 (0.56)\\
$10^{-12}$  &  3.086e-01 & 1.921e-01 (0.62)& 1.137e-01 (0.59)&
6.533e-02 (0.57) & 3.681e-02 (0.56)\\
\hline
&\multicolumn{5}{|c|}{$\|u^*-u^N\|_{\ell_{\infty}}$  (Reduction Rate w.r.t. N)} \\
\hline
$\eps$/$N$&32&64&128&256&512\\
\hline
$10^{-6}$  &  6.935e-02 & 1.981e-02 (0.29) & 6.436e-03 (0.32)&
2.051e-03 (0.32) & 6.448e-04 (0.31)\\
$10^{-8}$  &  6.945e-02 & 1.983e-02 (0.29)& 6.444e-03 (0.32)&
2.053e-03 (0.32) & 6.455e-04 (0.31)\\
$10^{-10}$  &  6.946e-02 & 1.984e-02 (0.29)& 6.445e-03 (0.32)&
2.054e-03 (0.32) & 6.456e-04 (0.31)\\
$10^{-12}$  &  6.946e-02 & 1.984e-02 (0.29)& 6.445e-03 (0.32)&
2.054e-03 (0.32) & 6.456e-04 (0.31)\\
\hline
\end{tabular}
\end{table}

\begin{table}[p]
\centering
\caption{$\beta$-weighted norm and discrete max norm errors for model
  problem with biquadratic discretization.}
\label{tab:biquad}
\begin{tabular}{|c|ccccc|}
\hline
&\multicolumn{5}{|c|}{$\tnorm{\mathcal{U}^\ast-\mathcal{U}^N}{\beta,2}$  (Reduction Rate w.r.t. N)} \\
\hline
$\eps$/N&32&64&128&256&512\\
\hline
$10^{-6}$  &  9.307e-02 & 3.854e-02 (0.41)& 1.394e-02 (0.36)&
4.655e-03 (0.33) & 1.484e-03 (0.32)\\
$10^{-8}$  &  9.306e-02 & 3.854e-02 (0.41)& 1.394e-02 (0.36)&
4.656e-03 (0.33) & 1.485e-03 (0.32)\\
$10^{-10}$  &  9.306e-02 & 3.854e-02 (0.41)& 1.394e-02 (0.36)&
4.656e-03 (0.33) & 1.485e-03 (0.32)\\
$10^{-12}$  &  9.306e-02 & 3.854e-02 (0.41)& 1.394e-02 (0.36)&
4.656e-03 (0.33) & 1.485e-03 (0.32)\\
\hline
&\multicolumn{5}{|c|}{$\|u^*-u^N\|_{\ell_{\infty}}$  (Reduction Rate w.r.t. N)} \\
\hline
$\eps$/N&32&64&128&256&512\\
\hline
$10^{-6}$  &  1.512e-02 & 1.817e-03 (0.12)& 2.715e-04 (0.15)&
3.609e-05 (0.13) & 4.133e-06 (0.11)\\
$10^{-8}$  &  1.518e-02 & 1.823e-03 (0.12)& 2.730e-04 (0.15)&
3.622e-05 (0.13) & 4.145e-06 (0.11)\\
$10^{-10}$  &  1.519e-02 & 1.823e-03 (0.12)& 2.733e-04 (0.15)&
3.624e-05 (0.13) & 4.147e-06 (0.11)\\
$10^{-12}$  &  1.519e-02 & 1.823e-03 (0.12)& 2.734e-04 (0.15)&
3.625e-05 (0.13) & 4.147e-06 (0.11)\\
\hline
\end{tabular}
\end{table}

\begin{table}
\centering
\caption{$\beta$-weighted norm and discrete max norm errors for model
  problem with bicubic discretization.}
\label{tab:bicubic}
\begin{tabular}{|c|ccccc|}
\hline
&\multicolumn{5}{|c|}{$\tnorm{\mathcal{U}^\ast-\mathcal{U}^N}{\beta,2}$  (Reduction Rate w.r.t. N)} \\
\hline
$\eps$/N&32&64&128&256&512\\
\hline
$10^{-6}$  &  2.786e-02 & 7.800e-03 (0.28)& 1.748e-03 (0.22)& 3.419e-04 (0.20)& 6.185e-05 (0.18)\\
$10^{-8}$  &  2.785e-02 & 7.800e-03 (0.28)& 1.749e-03 (0.22)& 3.420e-04 (0.20)&  6.187e-05 (0.18)\\
$10^{-10}$  &  2.785e-02 & 7.800e-03 (0.28)& 1.749e-03 (0.22)& 3.420e-04 (0.20)&  6.187e-05 (0.18)\\
$10^{-12}$  &  2.785e-02 & 7.800e-03 (0.28)& 1.749e-03 (0.22)& 3.420e-04 (0.20)&  6.187e-05 (0.18)\\
\hline
&\multicolumn{5}{|c|}{$\|u^*-u^N\|_{\ell_{\infty}}$  (Reduction Rate w.r.t. N)} \\
\hline
$\eps$/N&32&64&128&256&512\\
\hline
$10^{-6}$  &  4.364e-03 & 6.989e-04 (0.16)& 9.807e-05 (0.14)& 1.148e-05 (0.12)&  1.198e-06 (0.10)\\
$10^{-8}$  &  4.370e-03 & 6.993e-04 (0.16)& 9.812e-05 (0.14)& 1.149e-05 (0.12)&  1.199e-06 (0.10)\\
$10^{-10}$  &  4.371e-03 & 6.994e-04 (0.16)& 9.813e-05 (0.14)& 1.149e-05 (0.12)&  1.199e-06 (0.10)\\
$10^{-12}$  &  4.371e-03 & 6.994e-04 (0.16)& 9.813e-05 (0.14)& 1.149e-05 (0.12)&  1.199e-06 (0.10)\\
\hline
\end{tabular}
\end{table}

\section{Conclusions}
\label{sec:conclusions}

In the paper, we propose and analyse a new weighted-norm first-order
system least squares methodology tuned for singularly perturbed
reaction-diffusion equations that lead to boundary layers.  The
analysis includes a standard ellipticity result for the FOSLS
formulation in a weighted norm, and shows that this norm is suitably
weighted to be considered a ``balanced norm'' for the problem.
Numerical results confirm the effectiveness of the method.  Future
work includes
completing the error analysis by proving the necessary interpolation error estimates, with respect to
$\tnorm{ \cdot }{\beta,2}$, investigating the observed
superconvergence properties,  generalizing the theory to
convection-diffusion equations,
 and investigating efficient
linear solvers for the resulting discretizations.

\bibliographystyle{splncs04}
\bibliography{fosls_refs}

\end{document}